\documentclass[a4paper,oneside,10pt]{amsart}
\usepackage{fullpage}
\usepackage{bm}
\usepackage[pdftex]{graphicx,xcolor}
\usepackage[pdftex]{hyperref}
\usepackage{ascmac}
\usepackage{amsmath}
\usepackage{amscd}
\usepackage{amssymb}
\usepackage{amsfonts}
\usepackage{amsthm}
\usepackage[all]{xy}
\usepackage{mathtools}
\usepackage{enumerate}
\usepackage{tikz}
\usepackage{hyperref}
\usepackage{thm-restate}
\usepackage[shortlabels]{enumitem}
\usepackage{tikz-3dplot}
\usepackage{appendix}
\usepackage{url}
\usepackage{enumitem}

\newtheorem{thm}{Theorem}

\newtheorem{lem}{Lemma}
\newtheorem{prop}{Proposition}

\newtheorem{obs}{Observation}

\newtheorem{rem}{Remark}
\newtheorem{que}{Question}

\DeclareMathOperator{\Span}{Span}

\title{A note on distinct volume subsets problem}
\author{Koki Furukawa}

\begin{document}

\maketitle
\date{}
\begin{center}
{\footnotesize
Department of Applied Mathematics, \\
Faculty of Mathematics and Physics, Charles University, Czech Republic \\
\texttt{koki@kam.mff.cuni.cz}\\
}
\end{center}

\begin{abstract}
For \(2\le a\le d+1\), what is the largest integer $H_{a,d} (n)$ such that every set of \(n\) points in \(\mathbb{R}^d\) with no \(a\) points on a common \((a-2)\)-flat contains a subset of $H_{a,d} (n)$ points whose determined \((a-1)\)-dimensional simplices have pairwise distinct \((a-1)\)-dimensional volumes?
We construct \(n\)-point sets that improve the best known upper bounds for \(H_{a,d}(n)\) in several cases of $a$ and $d$.

We also study a dual version of the problem. 
Let \(D_d(n)\) the maximum number such that for any arrangement of $n$ hyperplanes in general position in $\mathbb{R}^d$, we can always find a subset of \(D_d(n)\) hyperplanes for which all the $d$-dimensional simplices that they define have distinct $d$-dimensional volumes.
We improve the current known upper bound for \(D_d(n)\) and give the first nontrivial lower bound for $D_2(n)$ and $D_3 (n)$.

\end{abstract}

\vspace{-0.5em}

\section{Introduction}
We call a \(k\)-dimensional affine subspace of \(\mathbb R^d\) a \(k\)-flat.
For \(d\ge 2\), a finite point set \(X\subset \mathbb R^d\) with \(|X|\ge d+1\) is in \emph{general position} if no \(d+1\) points of \(X\) lie on a common hyperplane (that is, on a common \((d-1)\)-flat).
Similarly, we say that an arrangement of hyperplanes \(\mathcal H\) in \(\mathbb R^d\) with \(|\mathcal H|\ge d+1\) is in \emph{general position} if every \(d\) hyperplanes in \(\mathcal H\) meet in a single point and no \(d+1\) hyperplanes have a common point.

\subsection{Distinct volume subsets for point sets}
Given a set \(P\) of \(n\) points in general position in \(\mathbb R^d\), we ask how large a subset of \(P\) can always be found such that all \(d\)-dimensional simplices determined by this subset have distinct \(d\)-dimensional volumes.

\begin{que} \label{que1}
Let \(P\) be a set of \(n\) points in general position in \(\mathbb R^d\).
What is the largest integer \(H_d(n)\) such that every such set \(P\) contains a subset of size \(H_d(n)\) for which all \(d\)-dimensional simplices determined by that subset have distinct \(d\)-dimensional volumes?
\end{que}

More generally, for $a \leq d+1$, we can ask the quantity $H_{a,d} (n)$, the maximum number $t$ such that any set of $n$ points in $\mathbb{R}^d$ with no $a$ points lying on the same $(a-2)$-flat contains a subset of $t$ points for which all the $(a-1)$-dimensional simplices that they define have distinct $(a-1)$-dimensional volumes.
Note that the case $a=d+1$ coincides $H_d (n)$.
The problem is asked by Conlon et al.~\cite{conlon2015distinct} as a related problem of Erd\H{o}s~\cite{erdos1946sets}.
They showed the lower bound $H_{a,d} (n) = \Omega_{a,d} (n^{1/((2a-1)d)})$\footnote{Throughout the paper, we use $O_\alpha$ to denote the standard big-$O$ notation in which the implicit constant may depend on $\alpha$ and similarly $O_{\alpha,\beta}$ to denote
that the implicit constant may depend on $\alpha$ and $\beta$. We use an analogous convention for $\Omega_{\alpha,\beta}$ and $\Theta_{\alpha,\beta}$.} for $a < d+1$ and $H_d (n) = \Omega_d (n^{1/(2d+1)})$.

On the other hand, it is easy to see that a suitable subset of the integer lattice gives the following upper bound.

\begin{obs} \label{obs:h_{a,d}}
\[
H_{a,d}(n)=
\begin{cases}
O_d (n^{1/d}) & a=2,\\[2mm]
O(n^{2/3}) & a=3,\ d=3,\\[2mm]
O_d\!\left(n^{4(d+1)/(3d(d-1))}\right) & a=3,\ d\ge 4,\\[3mm]
O_{a,d}\!\left(n^{2(a-1)^2/(ad)}\right) & 4\le a\le d,\\[3mm]
O_d\!\left(n^{d/(d+1)}\right) & a=d+1.
\end{cases}
\]

\end{obs}

\begin{proof}
Consider the \(d\)-dimensional grid \([t]^d\).
For $a \geq 3$, let \(P\subset [t]^d\) be the maximum subset with no \(a\) points lying on the same \((a-2)\)-flat.
Then \(P\) determines
$
\binom{|P|}{a}
$
\((a-1)\)-dimensional simplices of nonzero \((a-1)\)-dimensional volume.
Estimating the size of \(P\) has been extensively studied as a generalization of the \textit{no-three-in-line problem}.
The current known lower bounds for \(|P|\) are given in the following table:

\[
\renewcommand{\arraystretch}{1.15}
\begin{tabular}{c|c}
 & lower bound for \( |P| \)\\
\hline
\(a=3,\ d=3\) & \(\Omega(t^2)\)~\cite{por2007no}\\
\hline
\(a=3,\ d\geq 4\) & \(\Omega_d (t^{\,d-2+2/(d+1)})\)~\cite{barany1998convex}\\
\hline
\(4\le a\le d\) & \(\Omega_{a,d} (t^{d/(a-1)})\)~\cite{ghosal2025subsets}\\
\hline
\(a=d+1\) & \(\Omega_d (t)\)~\cite{brass2003counting,lefmann2012extensions,roth1951problem}
\end{tabular}
\]

On the other hand, the number of distinct \((a-1)\)-dimensional volumes realized in \([t]^d\) is at most
$
O_{a,d}\!\left(t^{2(a-1)}\right)
$
if 
$a<d+1$,
and at most
$
O_d(t^d)
$
if 
$a=d+1$.
Therefore, for any subset \(Q\subset P\) such that all \((a-1)\)-dimensional simplices determined by \(Q\) have distinct \((a-1)\)-dimensional volumes, we must have
$
\binom{|Q|}{a}=O_{a,d}\!\left(t^{2(a-1)}\right)
$
when \(a<d+1\), and
$
\binom{|Q|}{d+1}=O_d(t^d)
$
when \(a=d+1\).
Hence, we obtain the following desired bounds for $a\geq 3$:

\[
\renewcommand{\arraystretch}{1.2}
\begin{array}{c|c}
& \text{upper bound for } |Q| \\
\hline
a=3,\ d=3 & O(n^{2/3}) \\
\hline
a=3,\ d\ge 4 & O_d\!\left(n^{4(d+1)/(3d(d-1))}\right) \\
\hline
4\le a\le d & O_{a,d}\!\left(n^{2(a-1)^2/(ad)}\right) \\
\hline
a=d+1 & O_d\!\left(n^{d/(d+1)}\right)
\end{array}
\]

For $a=2$, since, the $d$-dimensional grid $n^{1/d} \times \cdots \times n^{1/d}$ has $O_d (n^{2/d})$ distinct distances, we have $H_{2,d} (n) = O_d (n^{1/d})$.
\end{proof}

We provide a construction of $n$-point set giving an improved uppper bound for $H_{a,d} (n)$ in several cases of $a$ and $d$.

\begin{thm} \label{thm:h_{a,d}}
For all integers $a$ and $d$ with $2 \leq a \leq d + 1$, 
if $n \geq 2a$, then
we have
$$
H_{a,d}(n) \lesssim \sqrt{2an}.
$$
Especially, $H_d (n) \lesssim \sqrt{2(d+1)n}$.
\end{thm}

Note that the upper bound in Theorem~\ref{thm:h_{a,d}} is stronger than the one given in Observation~\ref{obs:h_{a,d}} in the cases \((a,d)=(3,3)\), \((3,4)\), and for each fixed \(a\ge 4\), for all \(d\) satisfying
$
a > (8+d+\sqrt{d^2+16d})/8 \sim d/4.
$

\subsection{Distinct volume subsets for hyperplane arrangements}
As a dual version of Question~\ref{que1}, 
Dam\'{a}sdi et al.~\cite{damasdi2020triangle} asked the following.

\begin{que}[\cite{damasdi2020triangle}]
An arrangement of $n$ hyperplanes in general position are given in $\mathbb{R}^d$. 
What is the maximum number $D_d (n)$ such that we can always find a subset of these hyperplanes of this size for which all the $d$-dimensional simplices that they define have distinct $d$-dimensional volumes?
\end{que}

Recently, the auther showed in \cite{furukawa2025simplex} that 
$D_2 (n) \ll n(\log n)^{-c}$ for some absolute constant $c > 0$,
and $D_d (n) \ll n e^{-(\log \log n)^{c_d}}$ for some $c_d \in (0,1)$ for $d \geq 3$
by observing that $D_d (n) < r_{d+2} (n)$ holds for all $d \geq 2$, where $r_k (n)$ is the size of the largest $k$-arithmetic progression-free subset of $[n]$.
We significantly improve these lower bounds.

\begin{thm} \label{thm:D_d}
For all $n \geq 2d +2$, we have 
$$
D_d (n) \lesssim \sqrt{2(d+1)n}.
$$
\end{thm}

To the best of our knowledge, no lower bound for $D_d (n)$ was known. However, Dam\'{a}sdi et al.~\cite{damasdi2020triangle} obtained a partial result showing that if an arrangement of lines in the plane contains no six lines that are tangent to a common conic, then we have $\Omega (n^{1/5})$ as a lower bound.
Their proof is based on the idea of Conlon et al~\cite{conlon2015distinct}. But, in order to adapt their argument to line arrangements in the plane, one needs to exclude the degeneracy that many lines are tangent to a common conic.
The reason is that the geometry of the condition of having fixed area is different in the point-set setting and in the line-arrangement setting. Indeed, for a set of points in general position in the plane, if two points $p,p'$ and a real number $\lambda>0$ are fixed, then the locus of points $x$ such that the triangle $pp'x$ has area $\lambda$ is the union of two lines parallel to the line $pp'$.
On the other hand, in the line arrangement setting, for two nonparallel lines $l,l'$ and for a fixed real number $\lambda>0$, any line $L$ such that the triangle formed by $l,l',L$ has area $\lambda$ is tangent to two fixed hyperbolas having $l$ and $l'$ as asymptotes. Therefore, in the line arrangement setting, one needs to impose a restriction on the number of lines tangent to a common conic.

We overcome this difficulty by observing that for any fixed hyperbola,
every sufficiently large family of tangent lines to it contains a reasonably
large subfamily such that all triangles determined by this subfamily have
distinct areas.

\begin{thm}\label{thm:D_2}
For a sufficiently large $n$, there exists a positive constant $c>0$ such that 
$$
D_2 (n) \geq c n^{1/10}.
$$
\end{thm}

By generalizing the idea of the proof of Theorem~\ref{thm:D_2}, we obtain the following lower bound for $d=3$.

\begin{thm}\label{thm:betterboundforD_3}
For a sufficiently large $n$, there exists a positive constant $c > 0$ such that 
$$
D_3 (n) \geq c n^{1/21}.
$$
\end{thm}

\section{Proof of Theorem~\ref{thm:h_{a,d}}}
In this section, we prove Theorem~\ref{thm:h_{a,d}} by constructing a set of $n$ points in $\mathbb{R}^d$ with no $a$ points on the same $(a-2)$-flat such that every subset of size $\sim \sqrt{2an}$ forms two $(a-1)$-dimensional simplices of the same volume.

\begin{proof}[Proof of Theorem~\ref{thm:h_{a,d}}]
We first assume that \(d\) is even.
Define a curve \(\gamma\) in \(\mathbb R^d\) by
$$
\gamma(t):=
(\cos(\lambda_1 t),\sin(\lambda_1 t),\dots,\cos(\lambda_{d/2} t),\sin(\lambda_{d/2} t)),
$$
where
$
0<\lambda_1<\cdots<\lambda_{d/2}
$
are distinct real numbers.
For an integer \(n\geq2a\), consider the \(n\)-point set
\[
X \coloneqq \left\{ \gamma\!\left(\frac{2m\pi}{n\lambda_{d/2}}\right) : m=1,\ldots,n \right\}.
\]
We choose gereric $\lambda_i$ so that \(X\) contains no $a$ points lying on the same $(a-2)$-flat.
For \(\Delta\in\mathbb R\), define an affine transformation
$
T_\Delta:\mathbb R^d\to\mathbb R^d
$
by
$
T_\Delta=\operatorname{Diag}(R_{\lambda_1\Delta},\dots,R_{\lambda_{d/2}\Delta}),
$
where \(R_\theta\) denotes the rotation matrix in the plane about the origin through angle \(\theta\).
Then
$
T_\Delta(\gamma(t))=\gamma(t+\Delta)
$
holds for every \(\Delta, t\in\mathbb R\).
Moreover, since each \(R_{\lambda_j\Delta}\) is an orthogonal matrix, \(T_\Delta\) is an isometry.
Hence, for any distinct \(t_1,\dots,t_a\in\{1,\dots,n-1\}\) and any \(\Delta\in\mathbb R\), the \((a-1)\)-dimensional simplices determined by
$
\gamma(t_1),\dots,\gamma(t_a)
$
and
$
\gamma(t_1+\Delta),\dots,\gamma(t_a+\Delta)
$
have the same $(a-1)$-dimensional volume.

Now set
$$
M(a)\coloneqq
\left\lceil
\frac{1+\sqrt{1+8a(n-1)}}{2}
\right\rceil + 1.
$$
Since we assume $n \geq 2a$, we have $M(a) \leq n$ and $\lceil \binom{M(a)}{2}/(n-1) \rceil \geq a$.
Choose an arbitrary subset \(X_{M(a)} \subset X\) with \(|X_{M(a)}|=M(a)\). 
By the pigeonhole principle, there exist a shift $\Delta > 0$ and distinct $t_1, \ldots, t_a$ such that 
$\{ \gamma (t_1), \ldots, \gamma (t_a) \} \subset X_{M(a)}$ and $\{ \gamma (t_1 + \Delta), \ldots, \gamma (t_a + \Delta) \} \subset X_{M(a)}$.
It follows that the \((a-1)\)-dimensional simplices determined by
$
\gamma(t_1),\dots,\gamma(t_a)
$
and
$
\gamma(t_1+\Delta),\dots,\gamma(t_a+\Delta)
$
have the same \((a-1)\)-dimensional volume.
Hence
$
H_{a,d}(n)<M(a)\lesssim \sqrt{2an}.
$

When \(d\) is odd, we consider the curve \(\gamma'\) in \(\mathbb R^d\) defined by
\[
\gamma'(t):=
(\cos(\lambda_1 t),\sin(\lambda_1 t),\dots,\cos(\lambda_{(d-1)/2} t),\sin(\lambda_{(d-1)/2} t),t).
\]
The same argument shows that the \(n\)-point set
$$
\left\{ \gamma'\!\left(\frac{2m\pi}{n\lambda_{(d-1)/2}}\right) : m=1,\ldots,n \right\}
$$
gives
$
H_{a,d}(n)<M(a).
$
\end{proof}

\section{Proof of Theorem~\ref{thm:D_d}}
In this section, we prove Theorem~\ref{thm:D_d}. 
The idea of the construction is similar to that in Theorem~\ref{thm:h_{a,d}}, but we use the moment curve.
\begin{proof}[Proof of Theorem~\ref{thm:D_d}]
Assume that $n \geq 2d+2$ and we consider the $d$-dimensional moment curve
$
\Gamma(t):=(t,t^2,\dots,t^d)\in\mathbb R^d
\,\, (t\in\mathbb R).
$
It has a nice property that no hyperplane in \(\mathbb R^d\) intersects $\Gamma$ in more than \(d\) points. (For the proof and some other beautiful applications of using the moment curve, see for example~\cite{matouvsek2003using,ziegler2012lectures}.)

Since the $m$th derivative of $\Gamma (t)$ is
\[
\Gamma^{(m)}(t)
=
\bigl(0,\dots,0,m!, (m+1)!t,\dots,\tfrac{j!}{(j-m)!}t^{\,j-m},\dots,\tfrac{d!}{(d-m)!}t^{\,d-m}\bigr),
\]
 $\Gamma'(t),\dots,\Gamma^{(d-1)}(t)$ are linearly independent. Hence for each $t\in\mathbb R$ we may define the hyperplane
$$
H_t:=\Gamma(t)+\Span\{\Gamma'(t),\Gamma''(t),\dots,\Gamma^{(d-1)}(t)\},
$$
where $\Span A$ is an affine span of the finite point set $A \subset \mathbb{R}^d$.
Then $H_t$ can be written as
$
H_t=\{\bm{x}\in\mathbb R^d:F_t(\bm{x})=0\},
$
where
$
F_t(x_1,\dots,x_d):=
\sum_{m=1}^d \binom{d}{m}(-t)^{d-m}x_m+(-t)^d.
$
For $\Delta\in\mathbb R$, we define an affine transformation $T_\Delta:\mathbb R^d\to\mathbb R^d, \,\, (x_1, \ldots, x_d) \mapsto (x'_1, \ldots, x'_d)$ by
$$
x_m':=\sum_{r=0}^{m}\binom{m}{r}\Delta^{\,m-r}x_r
\qquad
(m=1,\dots,d,\ \ x_0:=1).
$$
Since
$
\sum_{r=0}^{m}\binom{m}{r}\Delta^{\,m-r}t^r = (t+\Delta)^m
$
for all $t \in \mathbb{R}$, 
we obtain
$
T_\Delta(\Gamma(t))=\Gamma(t+\Delta).
$
Moreover, since the linear part of $T_\Delta$ is lower triangular with all diagonal entries equal to $1$, its determinant is $1$. Therefore, $T_\Delta$ preserves $d$-dimensional volumes.
Let $L_\Delta$ denote the linear part of $T_\Delta$. Differentiating the equality
$
T_\Delta(\Gamma(t))=\Gamma(t+\Delta)
$
repeatedly with respect to $t$, we obtain
$
L_\Delta\bigl(\Gamma^{(m)}(t)\bigr)=\Gamma^{(m)}(t+\Delta)
$
 for each $m=1,\dots,d-1$ and all $t,\Delta\in\mathbb R$.
Thus,
\[
\begin{aligned}
T_\Delta(H_t)
&=
T_\Delta(\Gamma(t)+\Span\{\Gamma'(t),\dots,\Gamma^{(d-1)}(t)\}) \\
&=
T_\Delta(\Gamma(t))
+
L_\Delta\bigl(\Span\{\Gamma'(t),\dots,\Gamma^{(d-1)}(t)\}\bigr) \\
&= \Gamma(t+\Delta) + \Span\{\Gamma'(t+\Delta),\dots,\Gamma^{(d-1)}(t+\Delta)\} \\
&= H_{t+\Delta}.
\end{aligned}
\]

Now we consider the arrangement of $n$ hyperplanes
$
\mathcal H:=\{H_1,H_2,\dots,H_n\} \subset \mathbb{R}^d.
$
Note that for each $t \in [n]$, $H_t \in \mathcal{H}$ is an osculating hyperplane at $(t, \gamma(t))$ and 
$H_t$ can be written as
$$
H_t=\{x \in \mathbb{R}^d:F_t(x)=0\},
\qquad
F_t(x)=\sum_{m=1}^d \binom{d}{m}(-t)^{d-m}x_m+(-t)^d.
$$

For distinct $t_1,\dots,t_d$, it's not hard to see that the matrix whose rows are the normal vectors of $H_{t_1},\dots,H_{t_d}$ has determinant
$
\left(\prod_{m=1}^d\binom{d}{m}\right)
\prod_{1\le i<j\le d}((-t_i)-(-t_j))\neq0.
$
Thus, every $d$ hyperplanes in $\mathcal H$ meet in a unique point.
Similarly, for distinct $t_1,\dots,t_{d+1}$, the matrix whose rows are the affine coefficient vectors of $F_{t_1},\dots,F_{t_{d+1}}$ has determinant
$
\left(\prod_{m=1}^d\binom{d}{m}\right)
\prod_{1\le i<j\le d+1}(t_j-t_i)\neq0.
$
Thus, no $d+1$ hyperplanes have a common point. Therefore, $\mathcal H$ is in general position.

Since $n \geq 2d+2$, we have $M(d+1) \leq n$ and $\lceil \binom{M(d+1)}{2}/(n-1) \rceil \geq d+1$.
By the pigeonhole principle, for any subset $\mathcal{H}_{M(d+1)} \subset \mathcal{H}$ of size $M(d+1)$, there exist $\Delta \in \{1, \ldots, n-1 \}$ and 
a subfamily $\mathcal{A} = \{H_{t_1}, \ldots, H_{t_{d+1}}\}$ of $\mathcal{H}_{M(d+1)}$ such that 
$\mathcal{A}' \coloneqq T_{\Delta} (\mathcal{A}) = \{H_{t_1 + \Delta}, \ldots, H_{t_{d+1} + \Delta}\} \subset \mathcal{H}_{M(d+1)}$.
Then, the $d$-dimensional simplices determined by $\mathcal A$ and $\mathcal A'$ have the same $d$-dimensional volume.
Thus, 
$
D_d(n) < M(d+1) \lesssim \sqrt{2(d+1)n}.
$
\end{proof}

\begin{rem}
Since the $d$-dimensional simplices determined by
$
\Gamma(t_1),\ldots,\Gamma(t_{d+1})
$
and
$
\Gamma(t_1+\Delta),\ldots,\Gamma(t_{d+1}+\Delta)
$
have the same $d$-dimensional volume for any distinct $t_1,\dots,t_{d+1} \in \mathbb{R}$ and $\Delta > 0$, 
a set of $n$ points 
$
P \coloneqq \{\Gamma(1),\dots,\Gamma(n)\}
$
proves 
Theorem~\ref{thm:h_{a,d}} for the case $a=d+1$.
\end{rem}

\begin{rem}
Let $P^{*}=\{h_p : p\in P\}$ be the arrangement of $n$ hyperplanes dual to the set of $n$ points $P=\{\Gamma(1),\dots,\Gamma(n)\}$ under the standard dual transformation
$
p=(p_1,\dots,p_d)
 \longleftrightarrow
h_p: x_d=p_1x_1+\cdots+p_{d-1}x_{d-1}+p_d.
$
Then the arrangement \(P^{*}\) also gives the desired construction for Theorem~\ref{thm:D_d}.
To see this, for each $t\in\{1,\ldots,n\}$, let $h_t$ be a hyperplane dual to the point $\Gamma(t)\in P$. For $\bm{x}=(x_1,\ldots,x_d)\in\mathbb R^d$, define a polynomial $q_{\bm{x}}(s):=s^d+x_{d-1}s^{d-1}+\cdots+x_2s^2+x_1s-x_d$. Then $h_t$ can be written as $h_t=\{\bm{x}\in\mathbb R^d:q_{\bm{x}}(t)=0\}$. For $\Delta\in\mathbb R$, define an affine transformation $T_\Delta:\mathbb R^d\to\mathbb R^d$ by $q_{T_\Delta(\bm{x})}(s)=q_{\bm{x}}(s-\Delta)$. We can see that $T_{\Delta}$ preserves $d$-dimensional volumes and $T_\Delta(h_t)=h_{t+\Delta}$. The same argument as in the proof of Theorem~\ref{thm:D_d} gives the same bound $D_d(n) < M(d+1)$ for $n\geq 2d+2$. 
\end{rem}

\section{Proof of Theorem~\ref{thm:D_2}}

We begin with the following fact.

\begin{prop}[\cite{damasdi2020triangle,furukawa2025simplex}] \label{prop:hyperbola}
Let $l_1$ and $l_2$ be non parallel lines in $\mathbb{R}^2$ and $\lambda>0$ fixed.
Then, every line that forms a triangle with $l_1$ and $l_2$ of area $\lambda$ is all tangent to one of two fixed disjoint hyperbolas which have $l_1$ and $l_2$ as asymptotes.
\end{prop}

A hyperedge-colouring of a $k$-uniform hypergraph is called \textit{$m$-good} if each $(k-1)$-tuple of vertices is contained in at most $m$ hyperedges of any particular colour. 
We call a subset of the vertex set a \textit{rainbow clique} if all hyperedge of its induced complete hypergraph have distinct colours.
We use the following Ramsey-type result.

\begin{prop}[\cite{conlon2015distinct,martinez2015sunflower}] \label{prop:conlon}
Let $H$ be an $k$-uniform hypergraph on $n$ vertices. Then, every $m$-good hyperedge-colouring of $H$ contains a copy of a rainbow clique of size at least 
$
\Omega_k ((n/m)^{1/(2k-1)}).
$
\end{prop}

\begin{lem} \label{lem:triangles by a branch}
Let $C$ be a branch of a hyperbola.
Then, any set of $N$ lines that are tangent to $C$ contains a subset of at least $\Omega (N^{1/5})$ lines for which all the triangles that they define have distinct area.
\end{lem}

\begin{proof}
Since every invertible affine transformation preserves ratios of areas and maps lines and conics to lines and conics, we may assume that $C$ is the branch
$C\colon xy=1$, $x>0$.

Fix two distinct tangent lines $l_1,l_2$ to $C$ and a real number $\lambda>0$. By Proposition~\ref{prop:hyperbola}, the set of all lines $l$ such that $l_1,l_2,l$ determine a triangle of area $\lambda$ is the set of tangent lines to a fixed hyperbola $\Gamma_{l_1,l_2,\lambda}$ whose asymptotes are $l_1$ and $l_2$.
We claim that $C$ and $\Gamma_{l_1,l_2,\lambda}$ have at most four common tangent lines. Indeed, in the dual plane, the sets of tangent lines to $C$ and to $\Gamma_{l_1,l_2,\lambda}$ are both conics, and two conics have at most four intersection points. Therefore, there are at most four tangent lines $l$ to $C$ such that $l_1,l_2,l$ determine a triangle of area $\lambda$.

Now let $\mathcal{L}$ be a family of $N$ tangent lines to $C$. Consider the complete $3$-uniform hypergraph on the vertex set $\mathcal{L}$, and colour each triple of distinct lines in $\mathcal{L}$ by the area of the triangle. By the above observation, this colouring is $4$-good. Hence, by Proposition~\ref{prop:conlon}, it contains a rainbow clique of size at least $\Omega((N/4)^{1/5})$.

\end{proof}

Lemma~\ref{lem:triangles by a branch} can also be proved using the following simple probabilistic argument.

\begin{proof}[Alternative proof of Lemma~\ref{lem:triangles by a branch}]
We can assume $C \colon xy=1$ with $x>0$.
Fix an arbitrary family $\mathcal{L}$ of $N$ tangent lines to $C$.
Let $\mathcal{L}_x \subset \mathbb{R}_{>0}$ be the set of the $x$-coordinates of the points of tangency of the lines in $\mathcal{L}$.
Then, for any subset $\mathcal{S} \subset \mathcal{L}$ with $|\mathcal{S}| \geq 4$, the condition that all triangles determined by $\mathcal{S}$ have distinct areas is equivalent to the following: for any two distinct (not necessarily disjoint) triples $I=\{i_1,i_2,i_3\}, \, J=\{j_1,j_2,j_3\} \in \binom{\mathcal{S}_x}{3}$, the triangle determined by $I$ and the triangle determined by $J$ have distinct areas,
where, $\mathcal{S}_x \subseteq \mathcal{L}_x$ is the set corresponding to $\mathcal{S}$.
Equivalently, if we define
$$
F(\{a,b,c\}) \coloneqq \Big(\frac{(a-b)(b-c)(c-a)}{(a+b)(b+c)(c+a)}\Big)^2,
$$
Then the condition is that $F(I) \neq F(J)$ for every two distinct triples $I$ and $J$.
Indeed, for $t>0$, let $\ell_t$ denote the tangent line to $C$ at the point $(t,1/t)$. Then the area of the triangle determined by $\ell_a,\ell_b,\ell_c$ is
$$ 
2 \Big|\frac{(a-b)(b-c)(c-a)}{(a+b)(b+c)(c+a)}\Big|.
 $$
We call a pair $(I,J)$ of distinct triples satisfying $F(I)=F(J)$ a \textit{bad pair}.

First, we count the number of bad pairs $(I,J)$ with $|I\cup J|=4$ in $\mathcal{L}_x$.
Fix $J$, and write $F(J)=K$ for some constant $K>0$. Let
$
I\cap J= \{u,v\}.
$
Then the remaining element ${x}=I\setminus J$ satisfies
$$
F(I)=\left(\frac{(x-u)(u-v)(v-x)}{(x+u)(u+v)(v+x)}\right)^2=K.
$$
Thus, there are at most four possible choices for $x$. Since there are $O(N^3)$ choices for $J$, the number of bad pairs of this type is at most $c_4 N^3$ for some constant $c_4>0$.
Similarly, it is not hard to see that the number of bad pairs $(I,J)$ with $|I\cup J|=5$ and $|I\cup J|=6$ in $\mathcal{L}_x$ is $c_2 N^4$ and $c_3 N^5$, respectively, for some constants $c_2, c_3 >0$.

Now, construct a random subset $\mathcal{R}_x\subset \mathcal{L}_x$ by choosing each element of $\mathcal{L}_x$ independently with probability $p\coloneqq cN^{-4/5}$, where $c>0$ will be chosen later.
Let $B$ be the number of bad pairs contained in $\mathcal{R}_x$.
The probability that a bad pair $(I,J)$ is contained in $\mathcal{R}_x$ is $p^{|I\cup J|}$.
Therefore, choosing $c>0$ to be sufficiently small so that $c_6 c^6 \leq c/4$ and taking $N$ sufficiently large, we have
$$
\mathbb{E}(B)= c_1 p^4 N^3 + c_2 p^5 N^4 + c_3 p^6 N^5 \leq cN^{1/5}/2.
$$
Hence,
$$
\mathbb{E}(|\mathcal{R}_x|-B) \geq cN^{-4/5} N-c N^{1/5}/2=cN^{1/5}/2.
$$
Thus, there exists a choice of $\mathcal{R}_x$ for which
$
|\mathcal{R}_x|-B \geq cN^{1/5}/2.
$
Repeatedly deleting one element from the union of a bad pair, we
obtain a subset $\mathcal{S}\subset \mathcal{L}_x$ of size at least $cN^{1/5}/2$ which contains no bad pair.
\end{proof}

\begin{proof}[Proof of Theorem \ref{lem:triangles by a branch}]
Let $\mathcal{L}$ be a family of $n$ lines in the plane in general position.
We may assume that $\mathcal{L}$ has no $m \coloneqq n^{1/2}$ lines tangent to a common branch of a hyperbola.
Indeed, otherwise, by Lemma~\ref{lem:triangles by a branch}, the branch $C$ admits a subset of $\Omega(m^{1/5})= \Omega(n^{1/10})$ tangent lines for which all the triangles they define have distinct areas.

Fix two lines $l_1,l_2\in \mathcal{L}$ and $\lambda>0$.
Let $\Lambda$ be a family of lines $l\in \mathcal{L}$ such that the triangle determined by $l_1,l_2,l$ has area $\lambda$.
Then each line of $\Lambda$ are tangent to one of two hyperbolas having $l_1$ and $l_2$ as asymptotes.
From the assumption and by Proposition~\ref{prop:hyperbola}, we have $|\Lambda|\leq 4m$.

Now colour the hyperedges of the complete $3$-uniform hypergraph with vertex set $\mathcal{L}$ according to the area of the triangle determined by each triple of lines.
Then this colouring is $4m$-good.
Hence, by Proposition~\ref{prop:conlon}, there exists a rainbow clique of size at least $n^{1/5}/(16m)^{1/5}=n^{1/10}/16^{1/5}$. This completes the proof.
\end{proof}

\section{Proof of Theorem~\ref{thm:betterboundforD_3}}

For a family of planes $\mathcal{H}$ in general position in $\mathbb{R}^3$, we call a subfamily $\mathcal{S}$ \textit{a distinct volume subset of $\mathcal{H}$} if the volumes of the $\binom{|\mathcal{S}|}{4}$ tetrahedra determined by $\mathcal{S}$ are all distinct.
Also, define $D_3 (\mathcal{H})$ to be the maximum size of a distinct volume subset of $\mathcal{H}$.
For a plane $H: ax + by + cz + d = 0$ in $\mathbb{R}^3$, the point $[a:b:c:d]$ in $\mathbb{RP}^3$ is uniquely determined. We call this \textit{the dual point of $H$} and denote it by $H^{*}$. Also, for a family of planes $\mathcal{H}$ in $\mathbb{R}^3$, we denote the set of their dual points by $\mathcal{H}^{*}$.
Furthermore, we denote by $V(\{ H_1, H_2, H_3, H_4 \})$ the volume of the tetrahedron determined by a family $\{ H_1, H_2, H_3, H_4 \}$ of $4$ planes in general position.
A surface $S$ in $\mathbb{R}^3$ is called a \textit{reciprocal cubic surface} if
$S$ can be mapped to the surface $\Sigma: xyz=1$ by an invertible affine transformation.
For a family of planes $\mathcal{H}$ in $\mathbb{R}^3$ in general position, define
$$
\Gamma := \max_{\mathcal{T} \in \binom{\mathcal{H}}{3}, \lambda > 0} | \{ H \in \mathcal{H} \backslash \mathcal{T} : V(\mathcal{T} \cup \{H \} ) = \lambda \}
$$
and call it \textit{the degree of $\mathcal{H}$}.
The following is a natural generalization of Proposition~\ref{prop:hyperbola}.
\begin{prop}[\cite{furukawa2025simplex}] \label{tangent plane}
Let $H_1,H_2,H_3$ be three planes in $\mathbb{R}^3$ meeting at a unique point, and fix $\lambda>0$. Then there exist two disjoint fixed reciprocal cubic surfaces $S_{+}$ and $S_{-}$ such that every plane $H$ that forms a tetrahedron with $H_1, H_2$ and $H_3$ of volume $\lambda$ is tangent to one of $S_{+}$ and $S_{-}$. Moreover, both $S_{+}$ and $S_{-}$ have $H_1,H_2,H_3$ as asymptotic planes.
\end{prop}

The following lemma is the key for the proof of Theorem~\ref{thm:betterboundforD_3} and also the $3$-dimensional version of Lemma~\ref{lem:triangles by a branch}.
\begin{lem} \label{lem:surface}
Let $S \subset \mathbb{R}^3$ be a reciprocal cubic surface. Then any family of $N$ tangent planes to $S$ in general position contains a distinct volume subset of size $\Omega (N^{1/14})$.
\end{lem}

\begin{proof}[Proof of Theorem~\ref{thm:betterboundforD_3}]
Let $\mathcal{H}$ be a family of $n$ planes in $\mathbb{R}^3$ in general position.
It suffices to show that $\mathcal{H}$ contains a distinct volume subset of size $\Omega (n^{1/21})$.
Let $\Gamma$ be the degree of $\mathcal{H}$. 
We colour each hyperedge of the $4$-uniform complete hypergraph with vertex set $\mathcal{H}$ by the volume of the tetrahedron it defines. Then, this colouring is $\Gamma$-good.
 Hence, by Proposition~\ref{prop:conlon}, there exists a rainbow clique of size $\Omega ((n/\Gamma)^{1/7})$. 
 Observing that a rainbow clique corresponds to a distinct volume subset, we have
$$
D_3 (n) = \Omega ((n/\Gamma)^{1/7}).
$$

Now, take $\mathcal{T} = \{ H_1, H_2, H_3 \} \in \binom{\mathcal{H}}{3}$ and $\lambda = \lambda_0 > 0$, achieving $\Gamma$. Then, by Proposition~\ref{tangent plane}, 
all $H \in \mathcal{H} \backslash \mathcal{T}$ satisfying $V(\mathcal{T} \cup \{H \} ) = \lambda_0$ are tangent to one of the two reciprocal cubic surfaces $S_{+}, S_{-}$. Without loss of generality, we may assume that at least $\lceil \Gamma/2 \rceil$ planes in $ \mathcal{H} \backslash \mathcal{T}$ are tangent to $S_{+}$. Then, by Lemma~\ref{lem:surface}, $\mathcal{H} \backslash \mathcal{T}$ contains a distinct volume subset of size $\Omega (\lceil \Gamma/2 \rceil^{1/14})$. Therefore,
$$
D_3 (n) \geq \Omega (\max \{ (n/\Gamma)^{1/7}, \Gamma^{1/14} \}) \geq \Omega (n^{1/21}).
$$
\end{proof}

We now prove Lemma~\ref{lem:surface}. For this, we need the following lemma.

\begin{lem} \label{lem:curve}
Let $\mathcal{H}$ be a family of $N$ planes in $\mathbb{R}^3$ in general position such that $\mathcal{H}^{*}$ is contained in an irreducible algebraic curve $C \subset \mathbb{RP}^3$ of degree at most $O(1)$. Then $\mathcal{H}$ contains a distinct volume subset of size $\Omega (N^{1/7})$.
\end{lem}
First, assuming this lemma, we prove Lemma~\ref{lem:surface}. Lemma~\ref{lem:curve} will be proved at the end.
\begin{proof}[Proof of Lemma \ref{lem:surface}]
Let $\mathcal{A}$ be a family of planes satisfying the assumptions of the statement. Let $\Gamma_0$ be the degree of $\mathcal{A}$.
Then, by Proposition~\ref{prop:conlon}, $D_3 (\mathcal{A}) = \Omega ((N/\Gamma_0)^{1/7})$.
Take $\{ A_1, A_2, A_3 \} \subset \mathcal{A}$ and $\lambda_0 > 0$ achieving $\Gamma_0$.
Then, by Proposition~\ref{tangent plane}, every $A \in \mathcal{A} \backslash \{ A_1, A_2, A_3 \}$ satisfying $V( \{ A_1, A_2, A_3 \} \cup \{A \} ) = \lambda$ is tangent to one of two disjoint reciprocal cubic surfaces $S_1, S_2$ having $A_1, A_2, A_3$ as asymptotic planes. Thus, without loss of generality, we may assume that there is a subfamily $\mathcal{A'}$ of $\mathcal{A} \backslash \{ A_1, A_2, A_3 \}$ of size $\lceil \Gamma_0/2 \rceil$ such that every plane in $\mathcal{A'}$ is tangent to $S_1$.

Since the planes $A_1, A_2, A_3$ are asymptotic planes of $S_1$, while none of them is asymptotic planes of $S$, we have $S \neq S_1$. 
Moreover, $S^{*}$ is also a surface. 
Thus, $S^{*}$ and $S_1^{*}$ have no common component, and hence
the intersection $S^{*} \cap S_1^{*}$ is an algebraic curve in $\mathbb{RP}^3$ of degree $O(1)$.
Furthermore, since every plane in $\mathcal{A'}$ is tangent to both $S$ and $S_1$,
we have $\mathcal{A'}^{*} \subset S^{*} \cap S_1^{*}$.
The curve $S^{*} \cap S_1^{*}$ can be decomposed into at most $O(1)$ irreducible components, so by the pigeonhole principle and Lemma~\ref{lem:curve},
we obtain $D_3 (\mathcal{A'}) = \Omega \left(\left(\frac{\lceil \Gamma_0 /2 \rceil}{O(1)}\right)^{1/7}\right) = \Omega (\Gamma_0^{1/7})$.
Therefore,
$$
D_3 (\mathcal{A}) \geq \Omega (\max \{ (N/\Gamma_0)^{1/7}, \Gamma_0^{1/7} \}) \geq \Omega (N^{1/14}).
$$
\end{proof}

\begin{proof}[Proof of Lemma \ref{lem:curve}]
Let $\mathcal{H}$ be a family of planes satisfying the assumptions of the statement.
We first show that $D_3(\mathcal{H}) \geq 8$ provided that $N$ is sufficiently large, which will be needed later.
For $\mathcal{T} = \{ H_k, H_l, H_m \}$, where $H_i: a_i x + b_i y + c_i z + d_i = 0 \, (i \in \{ k, l, m \})$, consider a plane $H: ax + by +cz + d = 0$ such that $\mathcal{T} \cup \{ H \}$ is in general position.
$$
\Delta_{klm}=
\det\!\begin{pmatrix}
a_k & b_k & c_k \\
a_l & b_l & c_l \\
a_m & b_m & c_m
\end{pmatrix},
\,
\Delta_{0klm}=
\det\!\begin{pmatrix}
a_k & b_k & c_k & d_k \\
a_l & b_l & c_l & d_l \\
a_m & b_m & c_m & d_m \\
a & b & c & d
\end{pmatrix},
\,
\Delta_{0ij}=
\det\!\begin{pmatrix}
a_i & b_i & c_i \\
a_j & b_j & c_j \\
a & b & c
\end{pmatrix}
\, (i, j \in \{ k, l, m \})
$$
Then the square of the volume of the tetrahedron determined by $\mathcal{T} \cup \{ H \}$ is given by
$$
V(\mathcal{T} \cup \{ H \})^2 = \frac{\Delta_{0klm}^6}{36 (\Delta_{klm} \Delta_{0kl} \Delta_{0km} \Delta_{0lm})^2}.
$$
Let
$$
P_\mathcal{T} [a:b:c:d] := \Delta_{0klm}^6, \quad Q_\mathcal{T} [a:b:c:d] := 36 (\Delta_{klm} \Delta_{0kl} \Delta_{0km} \Delta_{0lm})^2.
$$
We call $\mathcal{T} \in \binom{\mathcal{H}}{3}$ \textit{exceptional} if there exists a fixed constant $\lambda > 0$ such that
$V(\mathcal{T} \cup \{H \})^2 = \lambda$
for every $H \in \mathcal{H} \backslash \mathcal{T}$.
For fixed $\mathcal{T} = \{ H_1, H_2, H_3 \} \in \binom{\mathcal{H}}{3}$ and $\lambda > 0$, define
$$
F_{\mathcal{T}, \lambda} := P_\mathcal{T} [a:b:c:d] - \lambda Q_\mathcal{T} [a:b:c:d].
$$
Then $F_{\mathcal{T}, \lambda}$ is a homogeneous polynomial of degree at most $6$ on $\mathbb{RP}^3$.
Every plane $H: ax + by +cz + d = 0$ satisfying
$V(\mathcal{T} \cup \{ H \})^2 = \lambda$
also satisfies $F_{\mathcal{T}, \lambda} [a:b:c:d] = 0$.
Consider the complete $3$-uniform hypergraph with vertex set $\mathcal{H}$, and colour an edge $\mathcal{T} \in \binom{\mathcal{H}}{3}$ red if $\mathcal{T}$ is exceptional and blue otherwise.
Suppose that $\mathcal{R} \subseteq \mathcal{H}$ is a red clique. Then all tetrahedra determined by the $4$-subsets of $\mathcal{R}$ have the same volume.
Since the following observation holds, we must then have $|\mathcal{R}| \leq 4$.
\begin{obs} \label{lem:redclique}
The tetrahedra determined by a family of $5$ planes in general position in $\mathbb{R}^3$ cannot all have the same volume.
\end{obs}

\begin{proof}
Consider a family $\{H_1, \ldots, H_5 \}$ of $5$ planes in general position. By Shannon's theorem, $\{H_1, \ldots, H_5 \}$ contains two simplicial cells. Denote them by $T$ and $T'$. Without loss of generality, we may assume that $T$ is determined by $H_1, H_2, H_3, H_4$ and $T'$ is determined by $H_1,H_2, H_3, H_5$. After applying an affine transformation, we may assume that
$$
H_i : x_i = 0  \quad (i = 1, \ldots 3), \quad H_{4} : x_1 + x_2 + x_3 = 1.
$$
Then, it is trivial that $V(\{H_1,H_2,H_3,H_4\}) = 1/3!$.
Since $T'$ lies on the opposite side of $T$ with respect to each $H_i \, (i=1, \ldots d)$, there exist real numbers $r_1, r_2, r_3 > 0$ such that
$H_5 : x_1/r_1 + x_2/r_2 + x_3/r_3 = -1$.
Note that $r_1, r_2, r_3$ are pairwise distinct since $\{H_1, \ldots, H_5 \}$ is in general position.
Without loss of generality, we can assume 
$
r_1<r_2,\,r_1<r_3.
$
We now compute the volume of the tetrahedron determined by $H_2,H_3,H_4,H_5$.
Let its four vertices be
$
P_k
:=
\bigcap_{j\in\{2,3,4,5\}\setminus\{k\}} H_j,
\, (2 \leq k \leq 5).
$
A straightforward computation gives
$P_5 = e_1$, $P_4=-r_1e_1$, $P_3=-\frac{r_1(r_3+1)}{r_3-r_1}e_1 + \frac{r_3(r_1+1)}{r_3-r_1}e_3$,
and
$P_2=-\frac{r_1(r_2+1)}{r_2-r_1}e_1 + \frac{r_2(r_1+1)}{r_2-r_1}e_2$.
Therefore,
$$
V(\{H_2,H_3,H_4,H_5\}) =
\frac{1}{3!}
\left|
\det
\left(
P_5-P_4,\,
P_2-P_4,\,
P_3-P_4
\right)
\right| 
=
\frac{1}{3!} \cdot
\frac{
r_2r_3(1+r_1)^3
}{
(r_2-r_1)(r_3-r_1)
}.
$$
Since
$
0<r_2-r_1<r_2,
$
and
$
0<r_3-r_1<r_3,
$
we obtain
$
V(\{H_2,H_3,H_4,H_5\}) > (1+r_1)^3/3! = V(\{H_1,H_2,H_3,H_4\}).
$
\end{proof}

Fix a sufficiently large constant $B$. By the hypergraph Ramsey theorem, if $N$ is sufficiently large, then there exists a blue clique $\mathcal{B}$ of size $B$.
Fix $\mathcal{T} = \{ H_1, H_2, H_3 \} \in \binom{\mathcal{B}}{3}$ and $\lambda > 0$. 
Since $\mathcal{T}$ is not exceptional, there exists $H \in \mathcal{H} \backslash \mathcal{T}$ such that $V (\mathcal{T} \cup \{H \})^2 \neq \lambda$. Since $\mathcal{T} \cup \{ H \}$ is in general position, $F_{\mathcal{T}, \lambda}(H^{*}) \neq 0$.
Thus, the restriction of $F_{\mathcal{T}, \lambda}$ to $C$ is not identically zero. Using that $C$ is an irreducible algebraic curve of degree at most $O(1)$,
$$
|C \cap Z (F_{\mathcal{T}, \lambda})| \leq 6 \deg (C) = O(1).
$$

Now consider the complete $4$-uniform hypergraph with vertex set $\mathcal{B}$ and colour each hyperedge by the volume of the tetrahedron it determines. This colouring is $|C \cap Z(F_{\mathcal{T}, \lambda})|$-good. Hence, by Proposition~\ref{prop:conlon}, $\mathcal{B}$ contains a rainbow clique of size

$$
\Omega \left(\left(\frac{B}{|C \cap Z (F_{\mathcal{T}, \lambda})|}\right)^\frac{1}{7}\right) = \Omega (B^{1/7}).
$$
Thus, $D_3 (\mathcal{H}) = \Omega (B^{1/7})$.
In particular, if $N$ is sufficiently large, then $D_3 (\mathcal{H}) \geq 8$.

We are now ready to prove the main assertion.
Let $\mathcal{M}$ be a largest distinct volume subset of $\mathcal{H}$, and put $m := |\mathcal{M}| = D_3 (\mathcal{H}) \leq N$.
We take $N$ sufficiently large so that $m \geq 8$.
Then, for every $H \in \mathcal{H} \backslash \mathcal{M}$, one of the following holds:

\begin{enumerate}[label={Case(\Roman*):}]
    \item There exist $\mathcal{T} \in \binom{\mathcal{M}}{3}$ and $\mathcal{Q} \in \binom{\mathcal{M}}{4}$ such that $V(\mathcal{T} \cup \{ H \} ) = V(\mathcal{Q})$. 
    \item There exist distinct $\mathcal{T}_1 , \mathcal{T}_2 \in \binom{\mathcal{M}}{3}$ such that $V(\mathcal{T}_1  \cup \{ H \} ) = V(\mathcal{T}_2 \cup \{ H \})$.
\end{enumerate}

In Case(I), for each $H \in \mathcal{H} \backslash \mathcal{M}$, the number of pairs $(\mathcal{T}, \mathcal{Q})$ satisfying $V(\mathcal{T} \cup \{ H \} ) = V(\mathcal{Q})$ is at most $\binom{m}{3} \binom{m}{4} = O(m^7)$. Fix one such pair $(\mathcal{T}, \mathcal{Q})$, and let $\lambda = V(\mathcal{Q})^2$, $\mathcal{T} = \{ H_1, H_2, H_3 \}$, where $H_i: a_i x + b_i y + c_i z + d_i = 0$. For $H^{*} = [a:b:c:d]$, define
$$
F_{\mathcal{T}, \mathcal{Q}} (H^{*})  := \Delta_{0123}^6 - 36 \lambda^2 (\Delta_{123} \Delta_{012} \Delta_{013} \Delta_{023})^2.
$$
Then this is a homogeneous polynomial of degree at most $6$ on $\mathbb{RP}^3$. Moreover,
$$
\mathcal{S}_{\mathcal{T},\mathcal{Q}}^{*} := \{ H^{*} = [a:b:c:d] : V(\mathcal{T} \cup \{ H \} ) = V(\mathcal{Q}), H \in \mathcal{H} \} \subset Z(F_{\mathcal{T},\mathcal{Q}}).
$$
We claim that $C \nsubseteq Z(F_{\mathcal{T},\mathcal{Q}})$.
Indeed, suppose that $C \subseteq Z(F_{\mathcal{T},\mathcal{Q}})$. Since $m \geq 8$, we can choose $H \in \mathcal{M} \subseteq \mathcal{H}$ such that $H \notin \mathcal{T} \cup \mathcal{Q}$. By the assumption, $F_{\mathcal{T}, \mathcal{Q}} (H^{*}) = 0$, and since $\mathcal{T} \cup \{ H \}$ is in general position, $H$ satisfies $V(\mathcal{T} \cup \{ H \}) = V (\mathcal{Q})$, a contradiction.
Since $C$ is irreducible and has degree at most $O(1)$, Bézout's theorem gives
$$
|\mathcal{S}_{\mathcal{T}, \mathcal{Q}}^{*}| =|C \cap \mathcal{S}_{\mathcal{T}, \mathcal{Q}}^{*}| \leq |C \cap Z(F_{\mathcal{T},\mathcal{Q}})| = 6 \cdot O(1) = O(1).
$$
In Case(II), for each $H \in \mathcal{H} \backslash \mathcal{M}$, the number of pairs $(\mathcal{T}_1, \mathcal{T}_2)$ satisfying $V(\mathcal{T}_1  \cup \{ H \} ) = V(\mathcal{T}_2 \cup \{ H \})$ is at most $\binom{m}{3} \binom{m}{3} = O(m^6)$.
First, fix a disjoint pair $(\mathcal{T}_1,\mathcal{T}_2)$ with $\mathcal{T}_1 = \{ H_1, H_2, H_3 \}$ and $\mathcal{T}_2 = \{ H_4, H_5, H_6 \}$, where $H_i: a_i x + b_i y + c_i z + d_i = 0$. For $H^{*} = [a:b:c:d]$, define
$$
G_{\mathcal{T}_1, \mathcal{T}_2} (H^{*})  :=  \Delta_{0123}^6 (\Delta_{456} \Delta_{045} \Delta_{046} \Delta_{056})^2 - \Delta_{0456}^6 (\Delta_{123} \Delta_{012} \Delta_{013} \Delta_{023})^2.
$$
Then this is a homogeneous polynomial of degree at most $12$ on $\mathbb{RP}^3$. Moreover,
$$
\mathcal{T}_{\mathcal{T}_1,\mathcal{T}_2}^{*} := \{ H^{*} = [a:b:c:d] : V(\mathcal{T}_1 \cup \{ H \} ) = V(\mathcal{T}_2 \cup \{ H \}), H \in \mathcal{H} \} \subset Z(G_{\mathcal{T}_1,\mathcal{T}_2}).
$$

We claim that $C \nsubseteq  Z(G_{\mathcal{T}_1,\mathcal{T}_2})$.
Indeed, suppose that $C \subseteq  Z(G_{\mathcal{T}_1,\mathcal{T}_2})$. Since $|\mathcal{T}_1 \cup \mathcal{T}_2| \leq 6$, we can choose $H \in \mathcal{M} \backslash (T_1 \cup T_2)$. By the assumption, $G_{\mathcal{T}_1, \mathcal{T}_2} (H^{*}) = 0$, and since $\mathcal{T}_1 \cup \{ H \}$ and $\mathcal{T}_2 \cup \{ H \}$ are both in general position, we obtain $V(\mathcal{T}_1 \cup \{ H \}) = V(\mathcal{T}_1 \cup \{ H \})$, a contradiction.

Next, suppose that $\mathcal{T}_1$ and $\mathcal{T}_2$ are not disjoint.
If $|\mathcal{T}_1 \cap \mathcal{T}_2| = 1$, without loss of generality, let $\mathcal{T}_1 = \{ H_1, H_2, H_3 \}$ and $\mathcal{T}_2 = \{H_1, H_4, H_5 \}$. Then the polynomial to consider is
$$
G_{\mathcal{T}_1, \mathcal{T}_2} (H^{*})  :=  \Delta_{0123}^6 (\Delta_{145} \Delta_{014} \Delta_{015} \Delta_{045})^2 - \Delta_{0145}^6 (\Delta_{123} \Delta_{012} \Delta_{013} \Delta_{023})^2.
$$
If $|\mathcal{T}_1 \cap \mathcal{T}_2| = 2$, without loss of generality, let $\mathcal{T}_1 = \{ H_1, H_2, H_3 \}$ and $\mathcal{T}_2 = \{H_1, H_2, H_4 \}$. Then, the polynomial to consider is
$$
G_{\mathcal{T}_1, \mathcal{T}_2} (H^{*})  :=  \Delta_{0123}^6 (\Delta_{124} \Delta_{012} \Delta_{014} \Delta_{024})^2 - \Delta_{0124}^6 (\Delta_{123} \Delta_{012} \Delta_{013} \Delta_{023})^2.
$$
In either case, one can verify in the same way that $C \nsubseteq  Z(G_{\mathcal{T}_1,\mathcal{T}_2})$.
Therefore, by Bézout's theorem,

$$
|\mathcal{T}_{\mathcal{T}_1, \mathcal{T}_2}^{*}| = |C \cap \mathcal{T}_{\mathcal{T}_1, \mathcal{T}_2}^{*}| \leq 12 \cdot O(1) = O(1).
$$

Hence,

$$
N - m = |\mathcal{H} \backslash \mathcal{M}| = O(1) (O(m^7) + O(m^6)),
$$

and therefore $m = \Omega(N^{1/7})$.
\end{proof}

\medskip
\noindent {\bf Acknowledgement.} 
The author would like to thank G\"{u}nter Rote for pointing out the possibility of using the moment curve in the upper bound construction for $D_d(n)$.

The work was supported by grant no. 23-04949X of the Czech Science Foundation (GA\v{C}R) and the Charles University Grant Agency (GAUK) project number 378426.

\typeout{}
\bibliography{Distinct_volume_subsets_problem}
\bibliographystyle{plainurl}

\end{document}